\documentclass{article}
\newcommand{\shv}{\operatorname{Shv}}
\usepackage{tikz-cd}
\usepackage{mycommands}
\usepackage[all]{xy}
\newcommand{\ee}{\mathbb{E}}

\newcommand{\fuk}{\operatorname{Fuk}}
\newcommand{\I}{\mathcal{I}}

\newcommand{\N}{\mathcal{N}}
\newcommand{\sle}{\succeq}
\renewcommand{\S}{\mathcal{S}}

\newcommand{\dirlim}{\underset{\to}{\text{lim}}}
\usepackage{diagrams}
\usepackage{url}
\title{Categorical Logarithmic Hodge Theory, I}
\author{Dmitry Vaintrob}
\date{}
\begin{document}
\maketitle
\abstract{We write down a new ``logarithmic'' quasicoherent category $\operatorname{Qcoh}_{log}(U, X, D)$ attached to a smooth open algebraic variety $U$ with toroidal compactification $X$ and boundary divisor $D$. This is a (large) symmetric monoidal Abelian category, which we argue can be thought of as the categorical substrate for logarithmic Hodge theory of $U$. We show that its Hochschild homology theory coincides with the theory of log-forms on $X$ with logarithmic structure induced by $D$, and in particular, that the noncommutative Hodge-to de Rham sequence on $\operatorname{Qcoh}_{log}(U, X, D)$ recovers known log Hodge structure on the de Rham cohomology of the open variety $U$. As an application, we compute the Hochschild homology of the category of coherent sheaves on the infinite root stack of Talpo and Vistoli in the toroidal setting. We prove a derived invariance result for this theory: namely, that strictly toroidal changes of compactification do not change the derived category of $\operatorname{Qcoh}_{log}(U, X, D)$. The definition is motivated by the coherent object appearing in the author's microlocal mirror symmetry result \cite{vtor}. 

In this paper, the first in a series, we work over an algebraically closed field of characteristic zero. The next installment will develop the characteristic p and mixed-characteristic theories.
  }
\section{Introduction}
\subsection{Goals of this series}
In this paper and subsequent papers in this series we build a ``machine'' to transpose results of categorical and noncommutative Hodge theory into a logarithmic context. The main applications are as follows:
\begin{enumerate}
\item\label{it:hodge1} Interpret logarithmic structures on the Hodge theory of open varieties in terms of the Hochschild calculus of some category.
\item Re-interpret the ``integral Hodge theory'' structure \cite{bms} of Bhatt, Morrow and Scholze using Kaledin's noncommutative point of view \cite{kaledin1} on $p$-adic cyclic objects.
\item Introduce a logarithmic structure on the Hochschild homology of families of $A_\infty$ algebras over a base with suitably logarithmic behavior at the boundary: this should include Fukaya categories, viewed as families over the Novikov base. (This would in particular imply regularity of the noncommutative Gauss-Manin connection for such families of categories: see also \cite{pvv}, where this is proven for smooth, proper algebras.)
\item Write down a ``logarithmic Tate-Shafarevich'' motive $B\bar{\ee}_{log},$ for $\bar{\ee}$ the universal nodal elliptic curve, with the classifying space of the nodal fiber interpreted in a logarithmic way.
\item\label{it:microlocal} Write down a candidate for the mirror on the $B$ side for ``microlocal'' versions of SYZ mirror symmetry. 
\end{enumerate}
In this paper, we address parts \ref{it:hodge1}.\ and \ref{it:microlocal}.\ above. Namely, given an open variety $U$ with a choice of ``nice'' partial compactification $V$ and boundary divisor $D : = V\setminus U$, we write down a symmetric monoidal Abelian category $\qcoh_{log}$ of logarithmic coherent (log-coherent) sheaves $\qcoh_{log}(U, V, D)$ which acts as the ``substrate'' for logarithmic theory, and prove a Hochschild-Kostant-Rosenberg result for its Hochschild homology. The author's point of view on this category comes from a mirror symmetry (in fact, a ``coherent-constructible correspondence'') result, \cite{vtor}, and the geometric interpretation afforded by the mirror comparison is our main computational tool throughout the paper. The category constructed turns out to be a very close cousin of the \emph{parabolic category of the infinite root stack} defined by Talpo and Vistoli, \cite{talpo-vistoli}, and as such should have a generalization in a wider logarithmic context. Talpo and Vistoli's category is a category of sheaves on a certain stack (not of finite type). The category of logarithmic coherent sheaves is the category of coherent sheaves on an \emph{almost stack}: a putative geometric object which combines properties of a stack and of Faltings' \emph{almost geometry}, \cite{faltings} and \cite{gabber-romero}. 

\subsection{Logarithmic Hodge theory}
Say $U$ is a smooth algebraic variety of finite type. Its de Rham cohomology $H^*_{dR}(U)$ can be computed as the hypercohomology of the complex $(\Omega^*,d)$. If $U$ is closed, then homology groups of each sheaf of $k$-forms $H^*\Omega^k$ are finite-dimensional and moreover, the Grothendieck spectral sequence for the hypercohomology degenerates at the $E2$ term: this is the famous Hodge-to-de Rham degeneration theorem of Deligne-Illusie, \cite{deligne-illusie}. On the other hand if $U$ is open and smooth, then $H^*_{dR}(U)$ is finite-dimensional but the spectral sequence used to compute it is not. It is reasonable therefore to hope to ``reduce'' the a priori infinite-dimensional spaces $H^i\Omega^k(U)$ to a finite-dimensional piece, which is sufficient to recover $H^*_{dR},$ and perhaps also recover Hodge-to de Rham degeneration. Such a reduction is provided by \emph{log geometry}. Namely, fix a normal-crossings compactification $X$ of $U$ (something that can always be done in characteristic zero), with normal crossings boundary divisor $D : = X\setminus U$. Define the sheaf of \emph{log-one-forms} $\Omega^1_{log}(X, D)$ on $X$ to be one-forms $\xi$ which have at most first-order poles on curves transversal to $D$ (this condition can be imposed locally and multiplied by local functions on $X$, so $\Omega^1_{log}$ is a coherent sheaf). Define $\Omega^*_{log}(X, D)$ to be the exterior power algebra of $\Omega^1_{log}$ over the sheaf of rings $\oo$. Then $\Omega^*_{log}$ inherits a differential from $\Omega^*(U)$. Define the log Hodge homology, $H^{p, q}_{log}(X, D): = H^p\Omega^q_{log}(X, D)$ and $H^*_{\t{log-dR}}$ for the hypercohomology of $(\Omega^*_{log}, d)$ on $X$. Then we have the following results:
\begin{cthm}{Hablicsek, \cite{hablicsek}}
  \begin{enumerate}
  \item     $H^*_{\t{log-dR}}(X, D)\cong H^*_{dR}(X)$
  \item The log Hodge-to de Rham spectral sequence $H^{p, q}_{log}(X, D)\implies H^*_{\t{log-dR}}(X, D)$ degenerates. 
    \end{enumerate}
\end{cthm}
This theorem has a generalization to the case where the compactification $X$ has mild singularities (but is still in some sense ``log-smooth''), for example when it is \emph{toroidal}: i.e., the pair $(X, D)$ looks locally like a pair of the form $(X_\Sigma, D_\Sigma)$ for $\Sigma$ a toric fan, $X_\Sigma$ the corresponding toric variety and $D_{\Sigma}: = X_{\Sigma}\setminus T$ the toric boundary. More generally, there is a theory of \emph{log structure} on $X$ which admits a Hodge and a de Rham theory which generalizes these two cases.

For $X$ a proper variety, both the Hodge cohomology $H^{p, q}(X)$ and the de Rham cohomology $H^*_{dR}(X)$ can be interpreted (up to regrading) as Hochschild invariants in a certain category: namely, there are \emph{Hochschild-Kostant-Rosenberg} (HKR) isomorphisms $$\bigoplus_{q-p = k}H^{p, q}(X)\cong HH_*\qcoh(X)$$ expressing Hodge homology as a Hochschild homology invariant and $$\bigoplus_{i \equiv \epsilon \t{ mod } 2} H^i_{dR}(X)\cong HP_\epsilon\qcoh(X)$$ expressing de Rham homology as a periodic cyclic homology invariant, both in terms of the category of quasicoherent sheaves on $X$. Multiplication on both then comes from the symmetric monoidal structure on $\qcoh(X)$. 

Now it was observed by Talpo and Vistoli \cite{talpo-vistoli}, building on work of Borne and Vistoli \cite{borne-vistoli} that the log-Hodge homology groups can be naturally computed as Tor groups in a certain category $\qcoh_{par}(X, D)$ of \emph{parabolic sheaves}: sheaves on the ``infinite root stack'' (and this can be generalized for a certain larger class of log structures). However, it is not quite the case that $HH_*\qcoh_{par}(X, D)$ recovers log-Hodge homology.

In this paper we introduce a refinement of the Talpo-Vistoli category $\qcoh_{par}(X, D)$ for $X$ a toroidal compactification of $U = X\setminus D$. We define the {\bf log-coherent category}, $\qcoh_{log}(U, X, D)$, a symmetric monoidal category constructed in a similar way to $\qcoh_{par}(X, D)$ and in fact (in a suitable sense) contained in it, but with a subtle difference involving Faltings' \emph{almost mathematical formalism} \cite{faltings} (originally used in the context of $p$-adic Hodge theory). For this category, we prove log analogues of HKR isomorphisms in characteristic zero: namely
\begin{thm}\label{formstheorem}
  $$HH_*\qcoh_{log}(U, X, D) \cong \bigoplus H^{p, q}_{log}(U, X, D),$$ compatibly with differentials on both sides and
  $$HP_\epsilon\qcoh_{log}(U, X, D) \cong \bigoplus_{i = \epsilon\t{ mod }2} H_{dR}(U).$$
\end{thm}
Note that both $HH_*$ and $HP_*$ can be computed on the level of the derived category. We show that a class of transformations that is known to produce isomorphisms on log-forms also gives equivalences of derived symmetric monoidal categories. Namely, we say that a map of toroidal varieties with boundary $\tau:(U', X', D')\to (U, X, D)$ is a (strictly) toroidal modification if the map $U'\to U$ is an isomorphism and on the neighborhood of the fiber over any $x\in X$ it coincides with an equivariant map of toric varieties of equal dimension (viewed as varieties with boundary). We have the following invariance result.
\begin{thm}\label{invariance}
  For $\tau: (U', X', D')\to (U, X, D)$ a strictly toroidal modification, the derived pullback functor $$D^b\tau^*:D^b\qcoh_{log}(U', X', D')\to D^b\qcoh_{log}(U, X, D)$$ is a symmetric monoidal equivalence of symmetric monoidal dg categories.
\end{thm}
\noindent Note that this is not true in general on the level of Abelian categories, or for more general birational maps.

In future work, we will re-interpret log-polyvector fields as Hochschild cohomology of $\qcoh_{log}$, and interpret the Gauss-Manin connection in terms of this formalism, with regularity provided for free. Perhaps the most interesting application of this formalism is the mixed- and positive-characteristic local case. Here the key question is as follows: suppose that $K$ is a DVR with residue ring $k$ and ring of integers $\oo.$ We can view $\spec(\oo)$ as a partial compactification of $\spec(K),$ and for any variety $X$ over $K$, a model $X_\oo$ is a partial compactification (over $\spec(\oo)$) of $X$. The model is said to \emph{have semistable reduction} if the special fiber $X_k$ is a normal-crossings divisor in $X_\oo.$ It is a problem of immense interest to understand an analogue of log-Hodge theory for the pair $(X_\oo, X_k).$ The state of the art answer to this problem for $K$ a $p$-adic field comes from a paper by Bhatt, Morrow and Scholze, which uses perfectoid formalism to define ``integral $p$-adic Hodge theory'', a Dieudonn\'e module structure associated to such a pair which captures what should be its log Hodge theory (in particular, comparison results with \'etale cohomology). Now it is a result of Kaledin \cite{kaledin1} that the Hochschild homology theory of any category $\C$ over $\zz_p$ (more generally, any cyclic object over $\zz_p$) gives rise to a filtered Dieudonn\'e module. Thus, the cyclic module associated to $HH_*\qcoh_{log}(X_K, X_\oo, X_k)$ produces precisely a (filtered) Dieudonn\'e module, which we hope recovers the integral $p$-adic Hodge theory with significantly less work. We will treat the positive and mixed-characteristic cases in the next installment of this series. In the present paper, to make things easier we will restrict our attention to the characteristic-zero setting.
\subsection{Microlocal mirror symmetry}
The origin of this category, as well as our main computational tool, comes from the author's work on microlocal mirror symmetry, \cite{vtor}. Here, we get the following results:
  Let $(U, X, D)$ be the toroidal space with boundary, where $U = \gg_m^n$ is the complex torus and $X$ a toric variety. Let $S : = \frac{\rr^n}{\zz^n}$ be the real torus, which we view as a group, and let $\shv(S)$ be the category of all topological sheaves on $S$. 
  \begin{thm}
    There is an equivalence of categories $$D^b\qcoh_{log}(U, X, D)\cong D^b\shv(S).$$ Moreover, this is an equivalence of symmetric monoidal dg categories with tensor product on $D^b\qcoh_{log}(U, X, D)$ is intertwined with the convolution product on $\shv(S)$ with respect to the group structure on $S$.
  \end{thm}
Now $\shv(S)$ is a ``microlocal Fukaya category'' for the symplectic cylinder $T^*S (= \rr^n\times (S^1)^n)$, and any projective toric variety can be viewed via the moment map as a Strominger-Yau-Zaslow degeneration of a toric fibration over the $n$-dimensional disk. Microlocal Fukaya categories at the moment are defined for cotangent bundles \cite{nadler-zaslow} and are a work in progress for more general (exact) symplectic varieties: see \cite{kontsevich-soibelman} and \cite{nadler-pants}. Nevertheless, we make the following conjecture:
  Suppose that $\S$ is a Lagrangian torus fibration over a real base $B$ which is part of an SYZ model. In particular, the dual torus fibration $\S^\vee$ over $B$ has almost complex structure which we assume to be integrable, and endowed with a degeneration $\bar{\S}^\vee$ over a compactification $\bar{B}$ of the base. Write $\S^\vee_\t{sing}$ for the boundary $\bar{\S}^\vee\setminus \S$.
\begin{conj}
  In contexts where the microlocal Fukaya category $\fuk_\mu(\S)$ on $\S$ can be defined, there is an equivalence of derived categories,
  $$D^b\qcoh_{log}(\S^\vee, \bar{\S}^\vee, \S^\vee_\t{sing})\cong \fuk_\mu(\S).$$
\end{conj}

\section{Preliminaries}
\subsection{Spaces with boundary}
Suppose $k$ is an algebraically closed field of characteristic zero, and work in the category of schemes (at first, of finite type) over $k$. In this section, the main geometric object of interest will be a ``space with boundary'', i.e. a triple $(U, V, D)$ with $V$ a separated normal variety of finite type, $D\subset V$ a Cartier (or $\qq$-Cartier) divisor and $U = V\setminus D$ a smooth Zariski open. A priori, $V$ is not assumed to be compact: if it is compact, the triple $(U, V, D)$ is called a ``compactification triple''. When convenient, we will abbreviate the data $(U, V, D)$ of a space with boundary and just use $V$ with the interior and boundary $U, D$ implicit.

A map of spaces with boundary $\tau:(U', V', D')\to (U, V, D)$ is, for the purposes of this paper, a generically finite map of schemes $\tau:V'\to V$ with $\tau^{-1}(D)_{red} = D'$ (equivalently, $\tau^{-1}(U) = U'$), and such that $\tau\mid U'$ is \'etale. A map of spaces with boundary is said to be \emph{\'etale} if it is \'etale on all of $V'$: in particular, the \'etale maps to a triple $(U, V, D)$ form a site which coincides with the \'etale site on $V$. A map $\tau:(U', V', D')\to (U, V, D)$ is called \emph{ramified \'etale} if it is finite (and \'etale on $U$). A map of spaces with boundary is called a \emph{covering} if it is proper. Our construction of the category $\qcoh_{log}(U, V, D)$ will be \'etale local on $V$, and based on the interplay between \'etale (on $V$) and ramified coverings. The crucial result is the following:
\begin{lm}\label{lm:finite-extension}
  Fix a space with boundary $(U, V, D).$ Then for any \'etale map of interiors $\tau_0:U'\to U,$ there is a unique space with boundary $(U', V', D')$ and ramified \'etale map $\tau: (U', V', D')\to (U, V, D)$ extending $\tau_0.$ 
\end{lm}
\begin{proof}
  Say $\tau_0:U'\to U$ is an \'etale map. It suffices to work locally, so WLOG we assume $V$ is affine with ring of functions $\oo_V$. Now since $D$ is $\qq$-Cartier, $U$ is also (possibly after passing to smaller affine neighborhood) affine. A finite map is locally affine, so $U'$ is \'etale with ring of functions $\oo_U'$ an \'etale extension of $U$. Let $\oo_V'$ be the normal closure of $\oo_V$ in $\oo_U'$. Since \'etale maps are normal, this is a normal ring, giving a finite \'etale extension $\tau_1:V'\to V$. Since the pullback of a $\qq$-Cartier divisor is $\qq$-Cartier, the triple $(U', V', D')$ with $D'$ the pullback of $D$ satisfies our requirements. The normalcy requirement on $V'$ guarantees that this is the unique such extension.
\end{proof}
We can consider any space as a space with trivial boudary. There is an evident notion of product: if $(U, V, D), (U', V', D')$ are spaces with boundary, their product, abbreviated as $V\times V'$, is the space $(U\times U', V\times V', V\times D'\cup D\times V'.)$.

We could also work with \emph{formal} spaces with boundary, consisting of triples $(U, V, D)$ with $V$ a formal variety, $D\subset V$ a closed divisor locally defined by a single function, and $U = V\setminus D$ \emph{formally} taken to be the punctured neighborhood. Since there are some difficulties with defining a category of formal punctured varieties, we can formulate all the definitions above just in terms of the pair $(V, D)$. In particular, in this context we make the following definition:
\begin{defi}
  A map $\tau:(U', V', D')\to (U, V, D)$ of formal spaces with boundary is a proper map $\tau:V'\to V$ with $\tau^{-1}(D) = D'$ and with ramification locus contained scheme-theoretically in a finite thickening of $D$. It is \emph{ramified \'etale} if $V'\to V$ is finite. 
\end{defi}
In this case Lemma \ref{lm:finite-extension} becomes tautological, as we formally define an \'etale map $U'\to U$ in this context as a ramified \'etale map of triples $(U', V', D')\to (U, V, D).$ These naturally form a site, where $U'\times_U U''$ for $(U', V', D')$ and $(U'', V'', D'')$ two ramified \'etale triples defined as $(V''', D''')$ with $V'''$ the proper transform of $V'\times_V V''$. In terms of this definition, all our results for triples $(U, V, D)$ of finite type will directly extend to formal spaces with boundary.

\subsection{Normal-crossings and toroidicity}
We will be interested in spaces with boundary locally modeled on two basic examples.
\begin{enumerate}
\item The ``$k$-dimensional full normal-crossings triple'', $$NC_k : = (\aa^k_{f_k}, \aa^k, V_{f_k},)$$ with $f_k$ the defining polynomial of the cross of the hyperplane cross, $$f_k : = x_1\cdots x_k$$ (for $x_1, \dots, x_k, x_{k+1}, \dots, x_n\in \oo_{\aa^n}$ the canonical coordinates). A generalized normal crossings triple is a space with boundary of the form $NC_k \times \aa^{n-k}.$ 
\item The ``affine toric triple'' $T_\sigma : = (T, X, X\setminus T)$ for $X$ a $k$-dimensional affine toric variety with \emph{full-dimensional} toric fan $\sigma$ (a single open cone in the real vector space $X^*(T)^\vee_\rr$) and $T$ the torus acting on $X$ with open orbit. A \emph{toroidal triple} is a triple of the form $T_\sigma\times \aa^{n-k}$ (for $\sigma$ a $k$-dimensional toric cone). For a more general toric fan $\Sigma,$ a toric triple is the collection $T_\Sigma: = (T, X_\Sigma, X_\Sigma\setminus T)$ for $X_\Sigma$ the toric variety associated to $\Sigma.$
\end{enumerate}
Each irreducible component of the boundary divisor of a normal-crossings, respectively, toroidal scheme has naturally induced normal-crossings, respectively, toroidal structure, and this induces a natural stratification of $V$ in both the normal-crossings and the toroidal case by closed strata $V_j$ of codimension $j$ with $V_0 = V$ and $V_{j+1}$ inductively defined as the $n-j-1$-dimensional boundary of $V_j$. Define in this case $U_j : = V_j\setminus V_{j+1}$ the open strata: these are smooth. In both cases, there is a unique $0$-dimensional orbit, $V_k = : \{x_0\}$. 

\begin{defi} A \emph{formal} full normal-crossings, resp., affine toric triple is the formal neighbrohood near $x_0$ of a normal-crossings, resp., affine toric triple.
\end{defi}
\begin{defi}
 We say that a space with boundary $(U, V, D)$ is normal-crossings, respectively toroidal, if it is \'etale locally isomorphic to $NC_k\times \aa^{n-k}$ or $T_\sigma\times \aa^{n-k}$ with $\sigma$ a $k$-dimensional cone. (And similarly for formal schemes, in the \'etale topology for formal schemes.)
\end{defi}

\begin{defi}
  A map of spaces with boundary $(U', V', D')\to (U, V, D)$ is called a (strict) \emph{toroidal modification} if the map $U'\to U$ is an isomorphism and there is an \'etale cover $V_i$ of $V$ with $V_i': = V_i\times_V V'$ such that the map $V_i'\to V_i$ is isomorphic to the map of toric triples $$\one\times \beta:\aa^k\times X_\Sigma\to \aa^k\times X_\sigma,$$ for $X_\Sigma\to X_\sigma$ a proper equivariant map of toric varieties. 
A map of spaces with boundary $(U', V', D')\to (U, V, D)$ is called a (strict) \emph{normal-crossings modification} if it is a strict normal crossings modification and both $V, V'$ are smooth. 
\end{defi}
The simplest strict normal-crossings modification is the blow-up of $V$ at a $k$-dimensional normal-crossings stratum.
\begin{defi}
  A map of spaces with boundary $(U', V', D')\to (U, V, D)$ is called a \emph{weak} toroidal, resp., normal-crossings, modification if it is an isomorphism on $U$ and \'etale locally the map $(U', V', D')\to (U, V, D)$ becomes a toroidal, resp., normal-crossings modification after extending $D$ to a larger (toroidal, resp., normal-crossings) divisor $D_+$. 
\end{defi}
The basic example of a weak normal-crossings mofidication is the blow-up of $\aa^1\times NC^1$ (i.e., the triple $(\gg_m\times \gg_m, \aa^1\times \gg_m, \aa^1)$) at a point of the boundary. In some sense, if we are interested in ``asymptotic'' geometry of an open variety $U$, it is sufficient to study normal-crossings compactifications. Namely, it follows from Hironaka's resolution of singularities theorem \cite{hironaka} that any smooth open variety $U$ has a normal-crossings compactification $X$. The \emph{weak factorization} theorem, proved by Wlodarczyk, \cite{wlodarczyk}, implies that any two compactifications of $U$ are related by a sequence of weak toroidal modifications. In fact, it has a strengthening of the following form
\begin{thm}[Strong factorization \cite{akmw}]
  For any two normal-crossings compactifications $X, Y$ of $U$, there is an integer $N$ and a sequence of compactificatiosn $X_i, Y_i$ with $$X = X_0\from \dots \from X_N = Y_N\to Y_{N-1}\to Y_{N-2}\to \dots\to Y_0 = Y$$ such that each arrow is a weak normal-crossings modification.
\end{thm}
Moreover, we can take these arrows to be blow-ups. Unfortunately, strong normal-crossings (or indeed, strong toroidal) modifications are not sufficient for a factorization result: already the blow-up of $\aa^1\times NC^1$ above cannot be birationally factored into strong toric morphisms. 
\subsection{Ramification}
Recall that the \'etale fundamental group of a $k$-dimensional torus $T^k$ is $\hat{\zz}^k$. A generating collection of covers is the set of covers $\tau_N:T^{(N)}\to T$ with $t_i\mapsto t_i^N,$ with $N$ running over $\nn^\times.$ The \'etale fundamental group of $\aa^{n-k}$ is trivial, so we once again have a fundamental collection of covers $\tau_N\times \one:(T^k\times \aa^{n-k})^{(N)}\to (T^k\times \aa^{n-k}).$ Note that these covers are independent of choice of basis, as they are characterized by being the maximal connected \'etale covers of exponent dividing $N$.

Given a standard toroidal triple $(T\times \aa^{n-k}, T_\sigma\times \aa^{n-k}, D)$, the covers $(T^k\times \aa^{n-k})^{(N)}$ extend uniquely to ramified triples $(T_\sigma\times \aa^{n-k})^{(N)}.$ In the normal-crossings case, the corresponding maps are $$\tau_N:(x_1,\dots, x_k, x_{k+1}, \dots, x_n)\mapsto (x_1^N, \dots, x_k^N, x_{k+1}, \dots, x_n).$$
In the formal setting, any ramified cover of $(T_\sigma)_{\hat{x_0}}$ also factors through one of the $\tau_N$ (this can be seen e.g. over $\cc$ by noting that any ramified formal map can be extended to a tubular neighborhood). Now suppose that $(U, V, D)$ is a toroidal space with boundary and $\tau:(U', V', D')\to (U, V, D)$ is a ramified Galois covering, with (disrete) Galois group $\Gamma$. Let $x\in V$ be a point. Choose a preimage $\tilde{x}\in V',$ and write $$\Gamma_{\tilde{x}}\subset \Gamma$$ for the stabilizer of $\tilde{x}.$ Choose a toroidal local neighborhood $X_\sigma\times \aa^k$ of $\tilde{x}$, together with choice of toroidal coordinates on $\gg_m^k\times \aa^{n-k}.$ Then we get $$\pi_1(V'_{\hat{\tilde{x}}})\cong \hat{\zz}^k.$$ The group $\Gamma_{\tilde{x}}$ is a quotient of $\hat{\zz}^k$ by some cofinite group isomorphic (non-canonically) to $\prod_{i = 1}^k \zz/N_i.$ Write $|\Gamma|_x$, the ``ramification level of $\Gamma$ at $x$'', for the l.c.m. of the $N_i$. This is the maximal $N$ such that the map $V'\to V$ factors in any toroidal neighborhood $T_\sigma\times \aa^{n-k}$ of $x$ through the $N$-power cover $(T_\sigma\times \aa^{n-k})^{(N)}$, and in particular does not depend on choices. 

\subsection{Affine inverse limits}
We will use the index $N$ loosely as belonging to a countable filtered index poset $I$ over which we take inverse limits. In particular will implicitly allow replacing it by a filtered sub-index set when comparing different inverse limits if they are connected by maps on some generating sub-index sets. Later, we will use the specific index set of multiplicative positive integers, $I = \nn^\times,$ with $N\sle N'$ when $N\mid N'.$ Say $X$ is an algebraic variety.

\begin{defi}
  A \emph{tower} $X^{(*)}$ over $X$ is a collection of spaces $$X\from X^{(1)}\from X^{(2)}\from \cdots$$ indexed by a countable filtered index poset $I$. We say a tower is \emph{affine} if $X$ is affine and each $X^{(N)}$ is affine. We say a tower is \emph{locally affine} if $X^{(*)}\mid U$ is affine for every affine open $U\subset X$. 
\end{defi}
Given a locally affine tower $X^{(*)}$ over $X$, define its \emph{affine limit} to be the inverse limit of the diagram $X^{(*)}$ in the category of schemes (not necessarily of finite type) affine over $X$. This limit exists, and is presented by the \emph{direct} limit of sheaves of rings of quasicoherent sheaves, $\dirlim_I \oo^{(N)}$ over $X$. Similarly for spaces with boundary, define:
\begin{defi}
  A tower $(U^{(*)}, V^{(*)}, D^{(*)})$ over a space with boundary $(U, V, D)$ is a diagram of spaces with boundary indexed by $I$. It is \'etale, respectively, ramified \'etale if each $V^{(n)}\to V$ is \'etale, respectively, finite.
\end{defi}
A ramified \'etale map is finite, therefore affine. Given a tower of ramified \'etale spaces with boundary $(U^{(*)}, V^{(*)}. D^{(*)}),$ write $(U^{(\infty)}, V^{(\infty)}, D^{(\infty)})$ for the ``triple'' (in general, no longer of finite type) with $U^{(\infty)}, V^{(\infty)}$ defined as the corresponding affine inverse limits and $D^{(\infty)}$ the closed subvariety locally defined by the sheaf of ideals $\I_D^\infty: = \dirlim_I I^{(N)}.$ 

A type of tower of particular interest for us will be the tower $\Gamma^{(*)}$ where $\Gamma^{(*)}$ is a collection of finite groups: generally quotients of a profinite group of \'etale deck transforms. In this case, $\Gamma^{(\infty)}$ will no longer in general be discrete, but it will be a group (as the Hopf algebra structures on extend to a Hopf structure on the direct limit). The example that is most important for us is the collection of quotients of the (profinite completion of) the integers, $\Gamma^N = \zz/N.$ In this case we compute the (direct) system of rings of functions on the $\Gamma_N$ as $$\Gamma_\infty : = \dirlim_{N\in \nn^\times} k\left[\frac{1}{N}\zz\right]\cong k\left[\qq/\zz\right].$$ (Here we are using algebraic closedness: more generally, the two sides differ by a Tate twist.) 

\subsection{Almost local sheaves}
Suppose $X$ is a scheme of finite type and $X^{(\infty)}$ is an affine scheme (possibly of infinite type) with sheaf of functions $\oo_\infty$ over $X$. In particular, we assume that $X^{(\infty)}$ has an affine cover $X_i^{(\infty)}$ with affine intersections (in fact, this picture makes sense more generally in the quasi-separated context, see \cite{gabber-romero}). Let $\I$ be a sheaf of ideals in $\oo_\infty.$ Suppose $\I$ is flat, and \emph{idempotent} $\I^2 = \I\subset \oo$ (a phenomenon that does not occur on irreducible varieties of finite type). Further, suppose that we have the following identity:
$$\I\otimes \hom(\I, \F) = \I\otimes \F$$ for any $\oo_\infty$-module $\F$. Then we define the category of \emph{almost local quasicoherent sheaves} $\qcoh(X^{(\infty)})$ on $X^{(\infty)}$ (with respect to $\I$) as the full subcategory of $\qcoh(X^{(\infty)})$ consisting of sheaves $\F$ such that $\I \F\to \F$ is an isomorphism. We say a sheaf $\F$ is \emph{almost zero} if $\F\otimes \I \cong 0$. Write $\qcoh^{a.l.}(\oo_\infty),$ resp., $\qcoh^{a.z.}(\oo_\infty)$ for the almost local, respectively, almost zero, subcategories. The category of almost zero sheaves on $V$ is equivalent to the category of sheaves of modules on $\oo_\infty/\I.$ Note that unlike the case for the boundary of the finite-type scheme, the property of being a pushforward from the boundary on $\oo/\I$ is stable under extension, i.e.\ a Serre subcategory. In particular, the category of almost-zero sheaves $\qcoh^{a.z.}(V^{(\infty)})$ is a Serre subcategory.  We have a pair of adjoint functors $$\t{forg}:\qcoh^{a.l.}(X^{(\infty)})\leftrightarrows \qcoh(X^{(\infty)}):-\otimes \I,$$ and this adjunction is a piece of a Serre quotient sequence of functors $\qcoh^{a.z.}(\oo_\infty)\to \qcoh(\oo_\infty)\to \qcoh^{a.l.}(\oo_\infty).$

If we have an affine algebraic group $\Gamma_\infty$ acting fiberwise on $\spec \oo_\infty$ and preserving $\I,$ the group $\Gamma_\infty$ acts on the (compactly generated) category $\qcoh(X^{(\infty)}).$ An equivariant object with respect to this action is a quasicoherent sheaf $\F$ on $X^{(\infty)}$ together with a fiberwise trivialization of the pullback $\mu^*\F$ for $\mu:\Gamma_\infty\times X^{(\infty)}\to X^{(\infty)}$ the action map, satisfying an associativity condition (see e.g. section 2 of \cite{vtor}). Define the category $\qcoh^{a.l.}(X^{\infty})^{\Gamma_\infty}$ to be the category of all $\Gamma_\infty$-equivariant sheaves $\F$ on $X^{(\infty)}$ which satisfy the almost locality condition.

\section{Log-coherent sheaves}
We are ready to define our new quasicoherent category. The definition in this paper is stacky by nature, but because of our use of Faltings' almost category formalism, we need to be careful about sheaf properties which are taken for granted in the theory of ordinary stacks. Absent a good theory of almost stacks (which the author hopes will soon exist), we build the category up from a suitable system of \'etale open covers ``by hand''.
Let $(U, V, D)$ be a space with boundary. 
\begin{defi} Call a tower $(U^{(N)}, V^{(N)}, D^{(N)})$ of ramified Galois covers filtering the Galois group of $U$ an \emph{exhaustive} Galois ramified tower. 
\end{defi}
As before, $N$ belongs to some filtered poset which we implicitly allow replacing by a filtered sub-index set when comparing different towers. Fix an exhaustive Galois ramified tower. Let $x_0$ be a point of the open stratum $U$ and choose (noncanonically) a compatible system of lifts $x^{(N)}\in V^{(N)}$. Let $\Gamma_N$ be the group of deck transformations of $V^{(N)}$ relative to $x^{(N)}$. The inverse limit of the $\Gamma_N$ is the geometric Galois group of $U$. The rings of functions $\oo(\Gamma_N)$ form a directed system of (Hopf) algebras, and write $\oo_\infty^\Gamma : = \bigsqcup \oo(\Gamma_N)$. The coproduct and counit on the $\Gamma_N$ extend uniquely to a coproduct and counit on $\oo_\infty$, and we define the group scheme $\Gamma_\infty := \spec\oo_\infty^\Gamma.$ Similarly, note that the $V^{(N)}$ are affine over $V$ and for any affine $V'\subset V$ the sheaves of rings $\Gamma(\oo V^{(N)}, V')$ form a directed system of quasicoherent sheaves of
rings over $V$; these glue to an affine scheme $V^{(\infty)}$ over $V$. Because of uniqueness of normal extension, $\Gamma^{(N)}$ acts on $V^{(N)}$, and one can check that $\Gamma^{(\infty)}$ acts on $V^{(\infty)}.$ Note that neither $V^{(\infty)}$ nor $\Gamma_\infty$ depends on the particular sequence of $U^{(N)}$, so long as it is filtering (since both are colimits in the same category, in different filtering sub-diagrams of the same diagram). In particular, we can define canonically the category of equivariant quasicoherent sheaves, $\qcoh(V^{(\infty)})^{\Gamma_\infty}$ (this should be thought of as sheaves on the stack $V_{\t{orbi-}\,\infty} : = V^{(\infty)}/\Gamma_\infty$). The data of equivariance can be defined locally over any affine cover of $V$ (with a glueing condition). Note further that the sheaves associated to the reduced preimages $D^{(N)}$ of $D$ extend to a sheaf of ideals $\I_{D^{(\infty)}}$ on $V^{(\infty)}$. We define the category $\t{Pre}\qcoh_{\t{log}}$ of \emph{pre log-coherent} sheaves associated to the pair $(V, D)$ as the category of $\Gamma_\infty$-equivariant sheaves $\F$ on $V^{(\infty)}$ such that the map $\F\otimes \I\to\F$ is an isomorphism. 

Now given an \'etale open $V'$ in $V$, consider the (not necessarily exhaustive) Galois ramified tower $V'\times_V V^{(N)}$, with action by the groups $\Gamma_N$. We can refine these to an exhaustive tower $(V')^{(N)}$ of $V'$, with Galois groups $\Gamma'_N$ mapping surjectively to $\Gamma_N.$ Taking limits, this glues to a map $j_\infty:(V')^{(\infty)}\to V^{(\infty)}$ compatible with actions with respect to the map of group schemes $\Gamma'_\infty\to \Gamma_\infty.$ In particular, pullback along $j_\infty$ gives a restriction functor $$j_\infty^* : \qcoh((V')^{(\infty)})^{\Gamma_\infty}\to \qcoh(V^{(\infty)})^{\Gamma_\infty}.$$ We use the notation $\F\mid V'$ for $j_\infty^*(\F)$ when there is no ambiguity. Since $j^*\I = \I',$ this functor preserves the property of almost locality, i.e.\ if $\I\otimes \F \to \F$ is an isomorphism then $\I'\otimes j^*\F\to j^*\F$ is as well. Thus we get a restriction functor $$j_\infty^*:\t{Pre}\qcoh_{\t{log}}(V,D)\to \t{Pre}\qcoh_{\t{log}}(V', D).$$ We define the category $\qcoh(V, D)$ as the \'etale sheafification of this presheaf of categories. I.e., an object of $\qcoh_{\t{log}}(V)$ is defined as a collection of objects $\F_i$ of $\qcoh_{\t{log}}(V_i)$ on some cover $V_i$ of $V$, together with compatible isomorphisms $\F_i\mid (V_i\cap V_j)\cong \F_j\mid (\V_i\cap V_j)$ with the usual descent condition for sheaves of (Abelian) categories: that the natural automorphism of $\F\mid (V_i\cap V_j\cap V_k)$ is the identity map for any triple of indices $(i, j, k)$. We similarly define a category of \emph{pre-parabolic} sheaves $$\t{Pre}\qcoh_{par}(U, V, D): = \qcoh(X^{(\infty)})^{\Gamma_\infty},$$ without the almost locality condition. We want to consider the limit of categories corresponding to smaller and smaller refinements $\{V_i\}:$ but in fact, it turns out that it always suffices to pass to a finite (and, indeed, a Zariski) cover, consisting of opens which \emph{admit maximal ramification} along the divisor $D$, i.e., such that the ramification level becomes infinitely divisible as we go up in the tower. 
\begin{defi} 
We say that $V$ admits maximal ramification if there exists a tower of ramified covering $V^{(N)}$ indexed by $N\in \nn^\times$, and such that $V^{(N)}$ has level a multiple of $N$ at each point $x\in V$. We say such a tower is \emph{maximally ramified}.
\end{defi}

\begin{prop}
  Each point $x\in V$ has a Zariski neighborhood which admits maximal ramification.
\end{prop}
\begin{proof}
It suffices to take a toroidal neighobhood. 
\end{proof}
Now suppose the pair $(V,D)$ admits maximal ramification. Then (by basechange), so does the pair $(V', D')$ for any Zariski (or \'etale) neighborhood. We prove the following Lemma.
\begin{lm}
  Let $V^{(N,\t{ram})}$ be a Galois tower of ramified coverings of $V$ which attains maximal ramification. Let $\Gamma_N^{\t{ram}}$ be the corresponding tower of Galois groups. As before, let $V^{(\infty, \t{ram})}, \Gamma_\infty^{\t{ram}}$ be the inverse limit schemes in the locally affine sense, and let $\qcoh^{a.l.}(V^{(\infty, \t{ram})})$ be the category of sheaves $\F$ such that the map $\I^{(\infty, \t{ram})}\otimes\F \to \F$ is an isomorphism. Then the category $\qcoh^{a.l.}(V^{(\infty,\t{ram})})^{\Gamma_\infty\t{ram}}$ is canonically equivalent (as a symmetric monoidal category) to $\qcoh^{a.l.}(V^{(\infty)})^{\Gamma_\infty}$. 
\end{lm}
\begin{proof}
$V^{(\infty)}$ is a principal bundle over $V^{(\infty,\t{ram})}$ with respect to the group $$\Gamma_{\infty}^{\t{unram}}: = \ker(\Gamma_\infty\to \Gamma_\infty^{\t{ram}}).$$ It follows that we have an equivalence of the $\Gamma_\infty^{\t{unram}}$-invariants of the pre-log-coherent category with the category coming from the ramified tower: $$\qcoh(V^{(\infty)})^{\Gamma_\infty^{\t{unram}}}\cong \qcoh(V^{(\infty, \t{ram})}).$$ Now the short exact sequence of groups $\Gamma_{\infty}^{\t{unram}}\to \Gamma_\infty\to \Gamma_\infty^{\t{ram}}$ gives an identification $$\qcoh(V^{(\infty)})^{\Gamma_\infty}\cong \qcoh(V^{(\infty, \t{ram})})^{\Gamma_\infty^{\t{ram}}}.$$
(see e.g. \cite{vtor}, section 2). Further, it is clear that this equivalence is symmetric monoidal and takes the equivariant boundary ideal $\I_\infty\in \qcoh(V^{(\infty)})^{\Gamma_\infty}$ to its ramified analogue $\I_\infty^{\t{ram}}\in \qcoh(V^{(\infty, \t{ram})})^{\Gamma_{\infty}^{\t{ram}}}$. Thus, it induces an equivalence on the subcategories of almost local objects. 
\end{proof}
Now suppose that $V ( = (U, V, D))$ is a space with boundary which admits maximal ramification and $V_i$ is a Zariski (or an \'etale) cover of $V$. Then a maximally ramified tower $V^{(N)}$ over $V$ induces maximally ramified systems over the $V_i$, and it is clear that logarithmic categories defined in terms of restrictions of the same maximally ramified tower glue in either an \'etale or Zariski sense on $V$. Thus, the pre-log category of the space with boundary $V$ coincides with its log category (computed in terms of any cover, and hence also in the limit). Since the intersection of maximally ramifiable \'etale (or Zariski) subsets is maximally ramifiable, we see that the log-coherent category of an arbitrary space with boundary $V$ can be glued out of pre-log-coherent categories on any \'etale cover by maximally ramifiable opens. An analogous result holds for parabolic sheaves: when a space with boundary $V$ is maximally ramifiable, its category of pre-parabolic sheaves coincides with its category of parabolic sheaves, and these patch to give a definition of the \emph{parabolic coherent sheaf category} $$\qcoh_{par}(U, V, D)$$ for any space with boundary. This definition coincides with the notion in \cite{talpo-vistoli} of parabolic coherent sheaves on the infinite root stack (corresponding to the canonical log structure on the toroidal boundary). 

\section{Toric mirror symmetry results}
Here we describe in detail the log-coherent category for a toric variety. Let $(T, X, D)$ be a toric variety, viewed as a torus with toric boundary. Let $T=\spec(k[\M])$ be a torus with $\M$ a lattice, $\N : = \M^\vee$, and $\Sigma$ a fan in $\N_\rr^\vee$ (a set of cones $\sigma\in \M_\rr^\vee$). The $N$-power maps $\tau_N:T\to T$ can be extended to ``geomoetric Frobenius'' maps $\tau_N:X\to X$, which attain maximal ramification at every point of $X$. Write $X^{(N)}$ for the tower of toric varieties, compactifying \'etale tower of tori $T^{(N)}$. These correspond to lattices $\N^{(N)} : = N\N$ and $\M^{(M)} : = \frac{1}{N}\M.$ The space $X$ has an atlas by affines $X_\sigma$ for $\sigma\in \Sigma,$ with rings of functions $\oo_\sigma : = k[\M\cap \sigma^\vee].$ The preimages $X_\sigma^{(N)} = X^{(N)}\times_X X_\sigma$ in the tower have rings of functions $k[\frac{1}{M}\M],$ so that the affine limit $X_\sigma^{(\infty)} = X^{(\infty)}\times_X X_\sigma$ is affine with ring of functions $k[\M_\qq\cap \sigma^\vee]$ (here we are using that $\dirlim \frac{1}{N}\M\cap \sigma^\vee = \M_\qq\cap \sigma^\vee$). Now the Galois groups $\Gamma_N = (\zz/N)^n$ are, as affine schemes, given by $\spec k[\M/\frac{1}{N}\M]$, and the limit is $$\Gamma_\infty : = \spec \frac{\M_\qq}{\M}.$$ (Note that there is also an action of the larger group $T^{(\infty)} = \spec k[M_\qq]$, representing the ``pro-toric equivariance,'' and the pro-Galois $\Gamma_\infty$ action is then induced from the mapping $\Gamma_\infty\to T^{(\infty)}$ induced from the quotient map on groups in the other direction.) 

Now the category of $\Gamma_\infty$-equivariant spaces is (symmetric monoidally) equivalent (see \cite{vtor} section 2) to the category of $\frac{\M_\qq}{\M}$-\emph{graded} spaces. In particular, the action of $\Gamma_\infty$ on $X^{\infty}$ comes from $\frac{\M_\qq}{\M}$-graded (ring) structure on the $\sigma\cap \M_\qq,$ with grading via the composition $$(\sigma\cap \M_\qq)\to \M_\qq\to \frac{\M_\qq}{\M}.$$ The sheaf of ideals $\I$ is patched locally out of the (graded) $\oo_\sigma^{(\infty)}$-ideals $$\I_\sigma : = k[M_\qq\cap \mathring{\sigma}^\vee],$$ spanned by the interior of the corresponding cones. 

\subsection{Microlocal mirror symmetry}
The main result of \cite{vtor} (for rational fans) can be written as follows: let $S(\M)$ be the topological torus associated to $\M = \N^\vee.$ Let $\shv\left(S(M)\right)$ be the category of all topological sheaves on the torus with (bounded) derived category $$D^b\shv(S(\M)).$$ Then if $\Lambda\subset T^*\left(S(\M)\right)$ is a polyhedral coisotropic subset of the tangent space, there is a notion from microlocal analysis (see \cite{kashapira}) of ``derived topological sheaves with a singular support condition'' $$D^b\shv_\Lambda\left(S(\M)\right)$$ (usually, one considers constructible sheaves, but if $\Lambda$ is polyhedral, topological sheaves work as well). For $\Sigma$ a fan, define its ``support'' $|\Sigma|\subset \N$ to be the union of all its cones, $$|\Sigma| : = \sqcup_{\sigma\in \Sigma}\sigma\subset \N.$$

Using the fact that $S(\M)$ is a group and therefore has canonically trivial cotangent bundle, we have $T^*S(\M) \cong N_\rr\times S(\M).$ In terms of this condition, we define the polyhedral coisotropic subset $$\Lambda : = S(\M)\times |\Sigma| \subset T^*S(\M)$$, with $|\Lambda|$ over every point of $S(\M)$. Note also that if $X'\to X$ is a toric map of toric varieties, this corresponds to a map of cones $\Sigma'\to \Sigma$ (i.e., an inclusion of each cone $\sigma'$ in some cone $\sigma\in \Sigma$), and in particular an inclusion of support subsets, $|\Sigma'|\subset |\Sigma|\subset \N_\rr$. This inclusion is an equality if and only if $\Sigma'$ is a \emph{refinement} of $\Sigma,$ which is the case if and only if the map of toric varieties $X'\to X$ is \emph{proper}.
In \cite{vtor}, the author proves the following result. 
\begin{thm}[Main theorem of \cite{vtor}]\label{mainfromvtor}
There is an equivalence of derived categories $$D^b\qcoh_{log}(T, X, D)\cong D^b\shv_{\Delta}\left(S(\M)\right).$$ Further, if $\Sigma\to \Sigma'$ is a map of fans (corresponding to a toric map of toric varieties $X'\to X$) then we have an inequality on supports: $\Lambda\subset \Lambda'$, so $D^b\shv_{\Delta}$ is naturally a (fully faithful, dg) subcategory of $D^b\shv_{\Delta'}$. The inclusion functor then is compatible with the pullback functor on coherent categories $$D^b\qcoh_{log}(T, X', D')\to D^b\qcoh_{log}(T, X, D).$$
\end{thm}
The key consequence of this result is the following:
\begin{cor}
  Suppose $\beta:X'\to X$ is a proper toric map of toric varieties. Then derived pullback induces an equivalence of derived categories, $D^b\qcoh_{log}(X)\cong D^b\qcoh_{log}(X').$ 
\end{cor}
We deduce the following theorem. 
\begin{thm}\label{toroidal-inv}
  Given any strictly toroidal map $\beta:(U', X', D')\to (U, X, D)$, the pullback map on derived categories $D^b\beta^*:D^b\qcoh_{log}(U, X, D)\to D^b\qcoh_{log}(U', X', D')$ is an equivalence of derived categories. 
\end{thm}
\begin{proof}
First, suppose that $V'\to V$ is a map of toric varieties with $V$ affine and $U\subset V$ is an open subset. Let $U'$ be the preimage of $U$. Then we claim that the functor $D^b\qcoh_{log}(U)\to D^b\qcoh_{log}(U')$ is an equivalence of categories. This follows from the fact that the map $D^b\qcoh_{log}(V)\to D^b\qcoh_{log}(V')$ is fibered (locally) over $V$, and so the equivalence of derived categories is compatible with tensor product by any object of $\qcoh_{log}(V)$. In particular, $D^b\beta^*(\F)\otimes \oo_U\cong D^b\beta^*(\F\otimes \oo_U),$ which implies the lemma. 

Further, we note that $D^b\qcoh_{log}$ is a sheaf of dg categories in the \'etale topology: this follows from the corresponding statement for the derived category of (ordinary) quasicoherent sheaves, which is known. The theorem now follows from Theorem \ref{mainfromvtor} and the definition of strict toroidal maps.
\end{proof}
\begin{cor}
  Suppose $(U, X, D)$ is a normal-crossings space with boundary. Given any blow-up map $\beta:X'\to X$ along a smooth center normally embedded in one of the normal-crossings components, the functor $D^bj^*:D^b\qcoh_{log}(X)\to D^b\qcoh_{log}(X')$ is fully faithful.
\end{cor}
\begin{proof}
  For any such blow-up map we can, \'etale locally near any point $x\in X$, enlarge the divisor $D$ so that $\beta$ is a blow-up along a full normal-crossings component. We then use that derived full faithfulness can be checked locally, and that the pullback map $\qcoh_{log}(U, X, D)\to \qcoh_{log}(U', X, D')$ is fully faithful if $D\subset D'$ is a normal crossings sub-divisor.
\end{proof}
\begin{rmk}
  In our definition of strict toroidal map, we used \'etale neighborhoods. We could instead use formal neighborhoods: indeed, the two definitions agree when the compactification $X$ is compact. When $X$ is not compact, we have a diagram of Serre categories

  \begin{diagram}D^b\qcoh_{log}\big(X'\setminus X'_x\big)^{nilp}&\rTo &D^b\qcoh_{log}(X')&\rTo&D^b\qcoh_{log}(X')_{\hat{x}}\\
    \uTo&&\uTo&&\uTo\\
    D^b\qcoh_{log}\big(X\setminus X_x\big)^{nilp}&\rTo &D^b\qcoh_{log}(X)&\rTo&D^b\qcoh_{log}(X)_{\hat{x}}
\end{diagram}
which by quasicompact induction reduces the problem to checking equivalence of the map $D^b\qcoh_{log}(U, X, D)_{\hat{x}}^{nilp}\to D^b\qcoh_{log}(U', X', D')_{\hat{x}}^{nilp}$ of support subcategories, which follows from the result on toric varieties by observing that a support condition can be expressed as vanishing of tensor product by a certain sheaf on $X$.
\end{rmk}
\subsection{The universal nodal elliptic curve}
A particularly interesting example of this equivalence is as follows. Let $(\bar{\ee}, \ee, \ee_\infty)$ be the orbifold with boundary corresponding to the universal elliptic curve over the modular orbifold with boundary $(Y(1), X(1), \infty) = (\bar{\M_{1, 1}}, \M_{1, 1}, \infty)$. The modular orbifold has (ramified) Galois covers $X(N)$ classifying curves with level structure (a basis for the $N$-torsion lattice) and there is a universal family $\bar{\ee}(N)$ over $X(N)$, which can also be viewed as an orbifold with boundary; indeed, if $N\ge 3,$ these orbifolds will be honest spaces (so that the formalism of this paper applies without modification). Let $PGL_2(\qq_p)$ be the $p$-adic topological group and let $K_N\subset PGL_2(\qq_p)$ be the level-$N$ subgroup (which only depends on the $p$-adic valuation of $N$). It is well-known that the Hecke algebra $K_N\backslash PGL_2(\qq_p)/K_N$ acts on $X_N$ by correspondences. This action extends to an action of the algebra $K_N'\backslash GL_2(\qq_p)/K_N'$ on the open universal family $\ee(N)$. (Here $K_N'\subset GL_2(\qq_p)$ is a level group). However, this action does not extend to the compactification $\bar{\ee}_N$. The reason for this is that the fiberwise $\ell$-power map $[\ell]: \ee\to \ee$ does not extend to $\bar{\ee},$ but rather extends to a map $\bar{\ee}_{+\ell}\to \bar{\ee}$, where $\bar{\ee}_{+\ell}$ is the $\ell-1$-fold blow-up of the univeral elliptic curve at the node of $\ee_\infty$ (this is the nodal family classified by a map $X(1)\to \M_{1, \ell^2}$, with the $\ell^2$ points at the node arranged as sets of $\ell$th roots of unity in a necklace of $\ell$ copies of $\pp^1$). Theorem \ref{toroidal-inv} implies that the derived log-coherent category of $(\ee, \bar{\ee}_{+\ell}, \ee_{\infty, +\ell})$ is equivalent to that of $(\ee, \bar{\ee}, \ee_\infty),$ and in fact the action of the full Hecke algebra $K_N'\backslash GL_2(\qq_p)/K_N'$ can be reconstructed on a dg categorical level.

\section{Forms and Hochschild homology}
From this section onwards, we assume that the space with boundary $V$ has normal-crossings (rather than toroidal) compactification. Up to passing to derived categories, we can reduce to this case using Theorem \ref{invariance}. Recall that we defined the category $\qcoh_{par}$ of parabolic coherent sheaves as sheaves on the orbifold locally (on maximally ramifiable opens) given by $\qcoh(V^{(\infty)})^{\Gamma_\infty}.$ Now the category $\qcoh_{log}(V)$ is a subcategory of $\qcoh_{par}(V)$, hence is enriched in $\qcoh_{par}(V)$ (note that also, a fortiriori, it is enriched over the category of coherent sheaves on $V$). For a couple of sheaves $\F, \G\in \qcoh_{log}(V)$, write $\homu(\F, \G)\in \qcoh_{par}(V)$ and $\rhomu(\F, \G)\in D^b\qcoh_{par}(V)$, for these enriched versions of $\hom.$ Now while $\qcoh_{log}$ is not Morita equivalent to the category of modules over a ring, it nevertheless has a good notion of Hochschild homology and cohomology: these are defined more generally for any dg (or indeed spectrally enriched) category by Blumberg and Mandell in \cite{blumberg-mandell}. The homology is defined as a very large complex associated to a simplicial space with simplices parametrized by (multilinear combinations of) cycles of objects with composable morphisms 
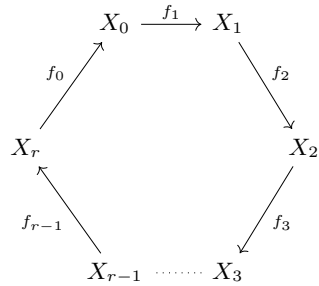
\[
\begin{tikzcd}[row sep=3.6em,column sep=1em]
& X_0\arrow[rr,"f_1"] && X_1 \arrow[dr,"f_2"] \\
X_r\arrow[ur,"f_0"]  &&  && X_2\arrow[dl, "f_3"] \\
& X_{r-1} \arrow[ul,"f_{r-1}"] && X_3 \arrow[ll, dotted, no head]
\end{tikzcd}
\]
(not necessarily commutative), and face and degeneracy morphisms come from composing edges or including a unit edge. Note that the diagram is cyclically asymmetric with the arrow $f_0$ playing a special role (corresponding to the basepoint in the standard cell decomposition of the circle). Here to make the definition precise we need to take a small cofibrantly generating piece of the category (for example, the derived category of countably generated modules should do the trick), and any two such choices produce canonically isomorphic homology, see \cite{blumberg-mandell}, so long as they both are contained in a third small subcategory. In particular, the $\qcoh_{log}$-enrichment of the arrow $f_0$ extends to an enrichment of the entire diagram, and we have a \emph{parabolic} Hochschild homology object $$CH_*^{par}(\qcoh_{log}V)$$ defined by the same diagram with the equivariant $V^{(\infty)}$ action taken into account. It is easy to see (analogously to the case of ordinary sheaves) that the Hochschild homology is \'etale local on the space with boundary $V$ (indeed, even formally local in an appropriate sense). Suppose WLOG $V$ is affine and maximally ramifiable. Define $T_\infty^\vee$ for the tangent sheaf of $V_\infty,$ given as a module by the quotient of $V_\infty^{\otimes 2}$ with image of $f\otimes g$ written $f dg$ with $f, g\in V_\infty$ and quotiented out by the standard relation $$x d(y_1 y_2) = xy_1 d y_2 + xy_2 dy_1.$$ This is a $\Gamma_\infty$-equivariant sheaf (it can be interpreted as $\tor_*^{V\times V}(\oo_\Delta, \oo_\Delta)$), and as it naturally lives in $V\times V$. Define $$(T_\infty^\vee)^{a.l.} : = \hom(\I, \I\otimes_{\oo_\infty}T_\infty^\vee).$$ This sheaf is naturally $\Gamma_\infty$-equivariant: write $T^\vee_{par}$ for the resulting object of $\qcoh_{par}$.

The main result of this section is the following theorem.
\begin{thm}
  \begin{enumerate}
  \item There is a natural map of complexes of parabolic sheaves $T_{log}^\vee[1]\to CH_*(\qcoh_{log})$.
  \item The map induced by graded commutative multiplication, $\Lambda^*T_{log}^\vee[1]\to CH_*^{par}(\qcoh_{log})$ is an isomorphism.
  \item Locally (for maximally ramifiable, affine $V$), we have $$R\Gamma\Big(V, CH_*^{par}\left(\qcoh_{log}(V)\right)\Big)\cong \Gamma(V, \Omega^*_{log}),$$ for $\Omega^*_{log}$ the sheaf of logarithmic differentials.
  \end{enumerate}
\end{thm}
To prove this, we compare to the internal Hochschild homology object $CH_*^{par}(\qcoh_{par}(V))$ of $\qcoh_{par}$ itself. On the one hand, we have a map $$T^\vee_{par}\to CH_*^{par}(\qcoh_{par}(V)),$$ induced by observing that $T^\vee(\Gamma^\infty) = 0$ in characteristic zero (indeed, a differential form on $\Gamma^\infty$ must factor through a differential form on some discrete $\Gamma^N$, hence is zero). On the other hand, we have a natural map $CH_*^{par}(\qcoh_{par})\to CH_*^{par}(\qcoh_{log}),$ given by taking a diagram of the form
\[
\begin{tikzcd}[row sep=3.6em,column sep=1em]
& X_0\arrow[rr,"f_1"] && X_1 \arrow[dr,"f_2"] \\
X_r\arrow[ur,"f_0"]  &&  && X_2\arrow[dl, "f_3"] \\
& X_{r-1} \arrow[ul,"f_{r-1}"] && X_3 \arrow[ll, dotted, no head]
\end{tikzcd}
\]
with $X_i\in \qcoh_{par}$ and replacing it with the diagram of $\I\otimes X_i.$ In particular, observe that for each such diagram, there is a map $\hom_{V^{(\infty)}}(\I, \hom(X_r, \I\otimes X_0))\to \hom(\I\otimes X_r, \I\otimes X_0)$ which is compatible with simplicial structure. This gives us a canonical map $\hom(\I, \I\otimes CH_*^{par}(\qcoh_{par}))\to CH_*^{par}(\qcoh_{log}V).$ Already on the level of $T^\vee_{par},$ this gives us logarithmic forms. To see this, up to application of an \'etale local change of coordinates we can assume that $V$ is a standard genearlized normal-crossings triple, $\aa^{n-k}\times NC_k.$ In fact, it will be enough for us to consider the simplest case $$V = NC_1 : = (\gg_m, \aa^1, \{0\}).$$ Here for any rational $q> 0$, we have a homogeneous form $\frac{1}{q}d x^{q} = x^{q-1}dx\in T^\vee_{par}.$ Now we have an obvious map $T^\vee_{par}\to \hom(\I, \I\otimes T^\vee_{par}).$ Now note that we have an additional map, which we call $d\log x:\I \to \I\otimes T^\vee_{par},$ which is the map sending $x^q\mapsto \frac{1}{q} d x^q$ for $q> 0$ (i.e. $x^q\in \I$). Thus we get a copy of $k[x^{\qq_+}]d\log x$ sitting inside $CH_*^{par}(\qcoh_{log})$. 

Now we use the following results that follow from arguments in \cite{vtor}.
\begin{thm}
  \begin{enumerate}
  \item   $HH_*\qcoh_{log}(\pp^1, \gg_m, \{0,\infty\}) \cong k\oplus k[1].$
  \item   $HH_*\qcoh_{log}(\aa^1, \gg_m, \{0\})\cong k[x]\oplus k[x][1].$
  \end{enumerate}

\end{thm}
The first result follows from a computation of the self-tor of the skyscraper sheaf of the (in fact, anti-\footnote{because we are conflating $\qcoh_{log}$ and $\qcoh_{log}^{op},$ which necessitates a sign.}) diagonal circle $\Delta \subset S^1\times S^1$, which represents the identity functor. For the second argument, let $\shv_+(S^1)$ be the category of sheaves with singular support the positive (co)tangent ray. Let $\Delta\in S^1\times S^1$ be the anti-diagonal. Then the category of dg functors $D^b\shv_+(S^1)\to D^b\shv_+(S^1)$ is classified by complexes of sheaves on $S^1\times S^1$ which have singular support contained in the quadrant $\rr_+\times \rr_+.$ The initial such object living over the (anti-)diagonal sheaf $\Delta$ is the pushforward of the constant sheaf on the semi-cone $$\frac{\{(x, y)\mid x-y> 0\}}{(x, y)\sim (x+1, y-1)}.$$ Now choosing tilted coordinates on the torus, this semi-cylinder becomes just a copy of $\rr\times S^1$, mapping to $S^1\times S^1.$ Using standard equivariance and push-pull arguments, the self-tor then becomes $\bigoplus_{k\in \zz}H_*(\rr_+\times S^1 \delta_{\rr_+ + k});$ the resulting space will be isomorphic to $k[x] H_*(S^1)$, where by tracing through the arguments, one can see that multiplication by $x$ on the two sides is intertwined. By using the same argument after quotienting out a finite subgroup of $S^1,$ we see that indeed, we have abstractly $CH_*^{par}(\qcoh_{log})\cong k[x^{\qq_+}].$ Thus we have a map $$d\log k[x^{\qq^+}]\to T^\vee_{log}\cong k[x^{\qq_+}],$$ which we know is injective and compatible with multiplication by $x^n$ by a restriction to $\gg_m$ argument. Further, we know that the two sides have a common finite-dimensional equivariant quotient (corresponding to the $\pp^1$ case). It follows that the map from the LHS to the RHS above is an isomorphism. We deduce the general case $D_1^k\times \aa^1$ from the case of $D_1$ and the space without boundary $\aa^1$ (where the statements are well-known) by applying a simple hybrid of the K\"unneth formula in topology and in algerbaic geometry. Now both sides embed injectively in $HH_*(\gg_m)$, and from this we deduce compatibility with differential. This completes our proof of the HKR isomorphisms, and implies Theorem \ref{formstheorem}.

\subsection{Computation of $HH_*(\qcoh_{par}(NC_1)).$}
  Using the above result and exactness of $CH_*$ for Serre short exact sequences of categories (see e.g. Blumberg, Gepner and Tabuada, \cite{blumberg-gepner-tabuada}), we can compute the Hochschild homology of the \emph{parabolic} category $HH_*(\qcoh_{par}(NC_1)).$ Namely, we have a Serre short exact sequence $\vect_{\qq/\zz}\to \qcoh_{par}(NC_1)\to \qcoh_{log}(NC_1),$ where $\vect_{\qq/\zz}$ is the category of $\qq/\zz$-graded vector spaces, with Hochschild homology $k^{\oplus \qq/\zz}.$ In particular, we see that we have an isomorphism $HH_{\ge 2}(\qcoh_{par}(NC_1))\to HH_{\ge 2}(\qcoh_{log}(NC_1)).$  On the level of $HH_1,$ we use that $\Gamma^\infty= \spec k[\qq/\zz]$ has trivial tangent bundle to see that $HH_1(\qcoh_{par}(NC_1))\cong T^\vee_{par}\cong k[x^{\ge -1}]dx,$ so that the map $HH_1(\qcoh_{par}(NC_1))\to HH_1(\qcoh_{log}(NC_1))$ is one-dimensional. By exactness, this leaves the remaining $k^{\oplus \qq/\zz\setminus 0}$ elements of $HH_0\vect_{\qq/\zz}$ to get added, so that $HH_0(\qcoh_{par}(NC_1))\cong HH_0(\qcoh_{log}(NC_1))\oplus k^{\oplus \qq/\zz\setminus 0}\cong x \cdot k[x]\oplus k^{\oplus \qq/\zz}.$ It should be possible to deduce from this the sheaf version of $HH_*(\qcoh_{par}(V))$ for any normal crossings space with boundary $V$ (by using a K\"unneth formula and \'etale locality), but the answer is much messier than the homology of $HH_*(\qcoh_{log}(V)).$

\nocite{*}
\bibliography{logbib}{}
\bibliographystyle{plain}

\end{document}